%% file: BOISNEAULT_BONAZZOLI_MARCHAND_CLAEYS_proceeding.tex
\documentclass[graybox]{svmult}

\usepackage{type1cm}        
\usepackage{makeidx}         
\usepackage{graphicx}        
\usepackage{multicol}        
\usepackage[bottom]{footmisc}

\usepackage{newtxtext}       %
\usepackage[varvw]{newtxmath}       

\makeindex             

\usepackage{mathtools}
\usepackage{cases}
\usepackage[capitalise,noabbrev]{cleveref}
\usepackage{bm}
\usepackage{bbm}
\usepackage{pifont} 
\usepackage{mathrsfs} 
\usepackage{caption}
\usepackage{subcaption}

\input{raccourcis.tex}
\usepackage{todonotes}
\usepackage{url} 

\begin{document}

\title*{Spurious resonances for substructured FEM-BEM coupling}
\author{
    Antonin Boisneault\inst{1,2}\orcidID{0000-0001-7986-9048}\and\\
    Marcella Bonazzoli\inst{2}\orcidID{0000-0002-0284-5643}\and\\
    Xavier Claeys\inst{1}\orcidID{0000-0003-0826-6244}\and\\
    Pierre Marchand\inst{1}\orcidID{0000-0002-2522-6837}}
  
\authorrunning{A. Boisneault, M. Bonazzoli, X. Claeys, P. Marchand}  

\institute{
  \inst{1}
  POEMS, CNRS, Inria, ENSTA, Institut Polytechnique de Paris, 91120 Palaiseau, France,  
   \email{antonin.boisneault@inria.fr, xavier.claeys@ensta.fr, pierre.marchand@inria.fr}\\ 
  \inst{2}
  Inria, Unité de Mathématiques Appliquées, ENSTA, Institut Polytechnique de Paris, 91120 Palaiseau, France, 
  \email{marcella.bonazzoli@inria.fr}
}
%
%
\maketitle

\abstract*{ 
We are interested in time-harmonic acoustic scattering by an impenetrable obstacle in a medium where the wavenumber is constant in an exterior unbounded subdomain and is possibly heterogeneous in a bounded subdomain. The associated Helmholtz boundary value problem can be solved by coupling the Finite Element Method (FEM) in the heterogeneous subdomain with the Boundary Element Method (BEM) in the homogeneous subdomain. Recently, we designed and analyzed a new substructured FEM-BEM formulation, called Generalized Optimized Schwarz Method (GOSM)~\cite{ourPaper}. Unfortunately, it is well known that, even when the initial boundary value problem is well-posed, the variational formulation of classical FEM-BEM couplings can be ill-posed for certain wavenumbers, called \emph{spurious resonances}. In this paper, we focus on the Johnson-Nédélec and Costabel couplings and show that the GOSM derived from both is not immune to that issue. In particular, we give an explicit expression of the kernel of the local operator associated with the interface between the FEM and BEM subdomains. That kernel and the one of classical FEM-BEM couplings are simultaneously non-trivial.
}

\section{Introduction and definition of the problem}\label{sec:problem_definition}

When solving the Helmholtz equation in complex heterogeneous media, it
is of interest to decompose the domain according to the variation of
the wavenumber. Local problems in homogeneous subdomains can be
reformulated as equations set on their boundary (which is particularly
useful for unbounded subdomains), and the global problem is solved
with Finite Element Method - Boundary Element Method (FEM-BEM)
coupling techniques~\cite{BielakMacCamy1983EIP,Costabel1987SMC,JohnsonNedelec1980CBI}.
Recently, a substructured formulation, called \emph{Generalized
Optimized Schwarz Method} (GOSM), has been designed for bounded
domains, with weakly imposed boundary conditions,
see~\cite{Claeys2023NOS}, and extended in~\cite{ourPaper} to unbounded
domains and interface conditions arising from several FEM-BEM coupling
techniques. Unfortunately, FEM-BEM variational formulations can be ill-posed for
certain wavenumbers called \emph{spurious resonances}~\cite{SchulzHiptmair2022SRC}.
The combined FEM-BEM formulations from~\cite{HiptmairMeury2006SFB} eliminate the spurious resonances, but rely on a compact regularizing operator and their numerical implementation  might be complicated.
In this paper, we choose to focus on the  classical
Johnson-Nédélec~\cite{JohnsonNedelec1980CBI} and
Costabel~\cite{Costabel1987SMC} couplings and show that also the associated
GOSMs suffer from the same spurious resonances. For both couplings, we give
an explicit expression of the kernel of the local operator associated
with the interface between the FEM and BEM subdomains. This kernel
and the one of the corresponding classical FEM-BEM couplings are
simultaneously non-trivial.


For simplicity, we consider a three subdomains configuration. The impenetrable obstacle \(\Omega_{\text{O}} \subset \mathbb{R}^d\) (\(d=2,3\))
is assumed connected Lipschitz bounded, with a connected complement \(\Omega \coloneqq \mathbb{R}^d \setminus \overline{\Omega_{\text{O}}}\)
decomposed into two connected Lipschitz subdomains: \(\Omega_{\text{F}}\) bounded with \(\partial \Omega_{\text{O}} \subseteq \partial \Omega_{\text{F}}\)
and \(\Omega_{\text{B}}\) unbounded with bounded boundary, satisfying
\(\Omega_{\text{F}} \cap \Omega_{\text{B}} = \emptyset\) (see Figure~\ref{fig:geometries}).
We are interested in solving the following Helmholtz problem, which models time-harmonic acoustic wave propagation: find \(u\in H^1_{\text{loc}}(\Delta, \Omega)\) such that
\begin{align}\label{pbm:initial}
  \begin{cases}
    \begin{aligned}
      &- \Delta u - \kappa^2 u = f, \text{ in } \Omega,\\
      &\text{boundary condition on } \partial \Omega, \\
      &\text{Sommerfeld's radiation condition},
    \end{aligned}
  \end{cases}
\end{align}
where the wavenumber \(\kappa \colon \Omega \rightarrow \mathbb{R}_{>0}\) is constant in \(\Omega_{\text{B}}\) and \(f \in L^2(\Omega_{\text{F}}) \). See~\cite[Def.~2.6.1]{SauterSchwab2011BEM} for the definition of \(H^1_{\text{loc}}(\Delta, \Omega)\), and~\cite[Sect.~2.2]{MR1822275} for the Sommerfeld's radiation condition. Boundary conditions can be of Dirichlet, Neumann, Robin, or even mixed type, so \(\Omega_{\text{O}}\) can represent an impenetrable obstacle with a sound-soft or sound-hard boundary, for instance.
\begin{figure}[!ht]
  \centering
  \begin{subfigure}[t]{0.48\textwidth}
    \centering
    \includegraphics[scale=0.7]{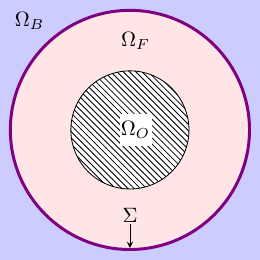} 
  \end{subfigure}
  \begin{subfigure}[t]{0.48\textwidth}
    \centering
    \includegraphics[scale=0.7]{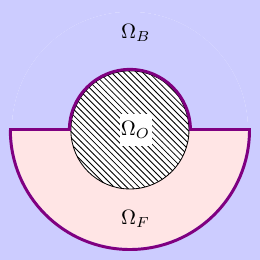} 
  \end{subfigure}
  \caption{Example of allowed (left) and forbidden (right) geometries (for simplicity). 
  }\label{fig:geometries}
\end{figure}

In FEM-BEM coupling, the FEM is used for the (possibly) heterogeneous subdomain \(\Omega_{\text{F}}\), the BEM is applied for the homogeneous domain \(\Omega_{\text{B}}\), and transmission conditions are imposed at the interface \(\Sigma \coloneqq\partial\Omega_{\text{F}}\cap\partial\Omega_{\text{B}}\). 
Our FEM-BEM formulation, based on the \emph{Generalized Optimized Schwarz Method} (GOSM)~\cite{ourPaper,Claeys2023NOS}, is set on \(\Sigma\) and reads:
\begin{align}\label{eq:substructurePb}
  (\Id+\Piop\Sop)
  (q_{\text{B}}, q_{\text{F}})
  = \textbf{rhs},
\end{align}
where \(\Piop\) is a (possibly) non-local exchange operator, \(\Sop \coloneqq \diag(\Sop_{\text{B}}, \Sop_{\text{F}})\) is a block-diagonal scattering operator and \(q_{\text{B}}, q_{\text{F}} \in H^{-1/2}(\Sigma)\) are outgoing impedance traces shared between \(\Omega_{\text{B}}\) and \(\Omega_{\text{F}}\).
More details about~\eqref{eq:substructurePb} can be found in~\cite[Prop.~8.1]{Claeys2023NOS}. 
To derive the substructured formulation~\eqref{eq:substructurePb}, the starting point is to write Problem~\eqref{pbm:initial} in variational form 
\begin{equation}\label{eq:var_pbm}
  \begin{aligned}
    & \text{Find}\ u\in H^{1}(\Omega_{\text{F}}), \ p \in H^{-{1}/{2}}(\Sigma)\ \text{such that } \forall v\in \mH^{1}(\Omega_{\text{F}}),  \forall q\in \mH^{-{1}/{2}}(\Sigma)\\
    & \int_{\Omega_{\text{F}}}(\nabla u \cdot \nabla v  - \kappa^{2} u v) \diff\bx + \langle \A_{\Sigma}(u\vert_{\Sigma},p),
    (v \vert_{\Sigma},q)\rangle = \ell(v),
  \end{aligned}
\end{equation}
where the explicit expression of \(\A_{\Sigma} \colon H^{{1}/{2}}(\Sigma)\times H^{-{1}/{2}}(\Sigma) \to H^{-{1}/{2}}(\Sigma)\times H^{{1}/{2}}(\Sigma)\) depends on the choice of the FEM-BEM coupling and involves {Boundary Integral Operators (BIOs)} (see~\cref{sec:bios}). The linear form \(\ell\) accounts for the contributions of the source term \(f\) and the boundary condition on \(\Omega_{\text{O}}\). For instance, \(\ell(v)= \int_{\Omega_{\text{F}}} f v \diff\bx\) for homogeneous Neumann boundary condition, or if \(\Omega_{\text{O}} = \emptyset\).
Here, the canonical duality pairing between a Banach space \(H\) and its topological dual \(H^*\) is denoted \(\langle \cdot, \cdot \rangle \colon H^{*} \times H \to \mathbb{C}\) and defined by $\langle \varphi, v\rangle\coloneq \varphi(v)$. 
We emphasize that the duality pairings we consider do \emph{not} involve any complex conjugation.

Next, we introduce a boundary operator \(\B_{\Sigma} \colon H^{{1}/{2}}(\Sigma)\times H^{-{1}/{2}}(\Sigma) \to H^{{1}/{2}}(\Sigma)\), defined by \(\B_{\Sigma}(\phi,p) \coloneqq \phi\), and also consider a \emph{transmission} (or \emph{impedance}) \emph{operator} \(\T_{\Sigma} \colon \! H^{{1}/{2}}(\Sigma) \to H^{-{1}/{2}}(\Sigma)\), satisfying \(\langle \T_{\Sigma}(\phi),\overline{\phi}\rangle > 0 \;\;\forall \phi\in H^{1/2}(\Sigma) \setminus \{0\}\), \(\T_{\Sigma}^* = \T_{\Sigma}\), and \(\Re(\T_{\Sigma}) > 0\).
The scattering operator \(\Sop_{\text{B}}\) involved in the GOSM is given by \(\Id
+ 2\, \imath \T_{\Sigma} \B_{\Sigma}(\A_{\Sigma}- \imath \B_{\Sigma}^* \T_{\Sigma} \B_{\Sigma})^{-1} \B_{\Sigma}^*\), see~\cite[Prop.~7.2]{Claeys2023NOS}. 
We aim to study the kernel of \(\A_{\Sigma} - \imath \B_{\Sigma}^* \T_{\Sigma} \B_{\Sigma}\), which is
crucial to establishing the well-posedness of the GOSM\@.

\section{Boundary integral operators and kernel of \(\A_{\Sigma} - \imath \B_{\Sigma}^* \T_{\Sigma} \B_{\Sigma}\)}\label{sec:bios}

The \emph{Dirichlet-Neumann trace map} on \(\Sigma\) (from \(\Omega_{\text{B}}\)), \(\gamma \colon H^1_{\text{loc}}(\Delta, \overline{\Omega_{\text{B}}}) \to H^{{1}/{2}}(\Sigma)\times H^{-{1}/{2}}(\Sigma)\), is defined as the unique bounded linear operator satisfying \(\gamma(\varphi) \coloneqq \left(\traceDir(\varphi), \traceNeu(\varphi)\right) \coloneqq \left(\varphi|_{\Sigma}, \symbvec{n}_{\text{B}} \cdot \nabla \varphi|_{\Sigma}\right), \forall \varphi\in \mathscr{C}^{\infty}(\overline{\Omega_{\text{B}}}) \coloneq\{\varphi\vert_{\Omega},\;\varphi\in\mathscr{C}^{\infty}(\mathbb{R}^d)\}\), where \(\symbvec{n}_{\text{B}}\) is the unit normal on \(\partial \Omega_{\text{B}}\) directed toward the exterior of \(\Omega_{\text{B}}\). 

Next, denote \(\mathscr{G}_{\kappa}\) the outgoing Helmholtz \emph{Green kernel} with wavenumber \(\kappa>0\), satisfying \((\Delta + \kappa^2) \mathscr{G}_{\kappa} = \delta\) in \(\mathbb{R}^d\) where \(\delta\) is the Dirac delta function. 
For \(d = 3\), \(\mathscr{G}_{\kappa}(\bx)\coloneqq \exp(\imath \kappa\vert \bx\vert)/(4\pi\vert \bx\vert)\), and for \(d = 2\) \(\mathscr{G}_{\kappa}(\bx)\coloneqq \imath H^{(1)}_{0}(\kappa \vert \bx\vert)/(4\pi)\), with \( H^{(1)}_{0}\) the \(0\)-th order Hankel function of the first kind~\cite[Chapter~9]{McLean2000SES}. For any \(\bx\in \mathbb{R}^d\setminus\Sigma\), and sufficiently smooth traces $(v,p)$,
define the \emph{total layer potential operator} by 
\begin{equation}\label{LayerPotentialOperator}
  \begin{aligned}
  \GL_{\Sigma}(v,p)(\bx)\coloneqq
  \int_{\Sigma}  \symbvec{n}_{\text{B}} (\by)\cdot(\nabla \mathscr{G}_{\kappa})(\bx-\by)\, v(\by) 
  +\;\mathscr{G}_{\kappa}(\bx-\by)\, p(\by) \diff s (\by),
  \end{aligned}
\end{equation}
where \(\diff s\) refers to the Lebesgue surface measure on \(\Sigma\). 
The map \((v,p)\mapsto \GL_{\Sigma}(v,p)\vert_{\Omega_{\text{B}}}\) can be extended by density as a bounded linear operator \(H^{1/2}(\Sigma)\times H^{-1/2}(\Sigma) \to H^{1}_{\text{loc}}(\Delta,\Omega_{\text{B}})\). 
For any pair \((v,p)\in H^{1/2}(\Sigma)\times H^{-1/2}(\Sigma)\), the function \(u \coloneqq \GL_{\Sigma}(v,p)\) solves the Helmholtz equation with wavenumber \(\kappa\) in \(\mathbb{R}^d \setminus \Sigma\), see e.g.~\cite[§3.1]{SauterSchwab2011BEM}, and satisfies Sommerfeld's radiation condition, see~\cite{MR1822275}.

Using the Dirichlet-Neumann trace map $\gamma$, we form the \emph{Calderón projector} \({\gamma\cdot\GL_{\Sigma}}\colon H^{1/2}(\Sigma)\times H^{-1/2}(\Sigma) \to H^{1/2}(\Sigma) \times H^{-1/2}(\Sigma)\), commonly decomposed as 
\begin{equation*}\label{eq:BIO_definition}
  \gamma\cdot\GL_{\Sigma} =   
  \dfrac{1}{2}\begin{bmatrix}
    \Id & 0\\
    0 & \Id
  \end{bmatrix}
  +
  \begin{bmatrix}
    \K_{\kappa} & \V_{\kappa} \\
    \W_{\kappa} & \Ktilde_{\kappa}
  \end{bmatrix},
\end{equation*}
where the four classical \emph{Boundary Integral Operators} (BIOs) appear: \(\V_{\kappa} \colon H^{-1/2}(\Sigma)\to H^{1/2}(\Sigma)\) (\emph{single layer}), \(\K_{\kappa} \colon H^{1/2}(\Sigma)\to H^{1/2}(\Sigma)\) (\emph{double layer}), \(\Ktilde_{\kappa} \colon H^{-1/2}(\Sigma)\to \mH^{-1/2}(\Sigma)\) (\emph{adjoint double layer}), and \(\W_{\kappa} \colon H^{1/2}(\Sigma)\to H^{-1/2}(\Sigma)\) (\emph{hypersingular}). Both \(\V_{\kappa}\) and \(\W_{\kappa}\) are symmetric, while \({\K_{\kappa}^* = - \Ktilde_{\kappa}}\). We recall that the adjoint \(*\) does not involve complex conjugation. 

These BIOs naturally arise when deriving Boundary Integral Equations for Dirichlet or Neumann boundary conditions. For instance, for \(u \in H^1(\Omega_{\text{B}})\) such that \(\Delta u + \kappa^2 u = 0\) in \(\Omega_{\text{B}}\), they appear in the \emph{Calderón equations} (see~\cite[Sect.~3.4]{SauterSchwab2011BEM})
\begin{subequations}
  \begin{align}
    (\Id/2 - \K_{\kappa}) \traceDir(u) &= \V_{\kappa} \traceNeu(u), \label{eq:calderon_1} \\
    (\Id/2 - \Ktilde_{\kappa}) \traceNeu(u) &= \W_{\kappa} \traceDir(u). \label{eq:calderon_2}
  \end{align}
\end{subequations}
Alternatively, combinations of these BIOs arise when dealing with an
outgoing impedance boundary condition \(\symbvec{n}_{\text{B}} \cdot \nabla
u|_{\Sigma} - \imath \T_{\Sigma} u |_{\Sigma} = g\), \(g \in
H^{-1/2}(\Sigma)\). For instance, subtracting \(\imath \V_{\kappa}
\T_{\Sigma} \traceDir(u)\) in Equation~\eqref{eq:calderon_1} yields \(
\DD_{\kappa, \T_{\Sigma}}^{*} \traceDir(u) = \V_{\kappa}(g) \), where
we introduced the \emph{impedance BIO} \(\DD_{\kappa, \T_{\Sigma}}^{*}
\coloneqq ({\Id}/{2} - \K_{\kappa} ) - \imath \V_{\kappa} \T_{\Sigma}
\).  Other impedance BIOs could be derived using the Calderón
equation~\eqref{eq:calderon_2}, but we do not focus on them.

For particular values of \(\kappa\), called \emph{spurious resonances},
\(\V_{\kappa}, \K_{\kappa}, \Ktilde_{\kappa}\) and \(\W_{\kappa}\) are well known to be
singular~\cite{ChandlerWildeGrahamEtAl2012NAB,McLean2000SES,SauterSchwab2011BEM}. A similar result holds for the impedance BIO \(\DD_{\kappa, \T_{\Sigma}}^{*}\), see~\cite[Sect.~2.6]{ChandlerWildeGrahamEtAl2012NAB}.
\begin{lemma}\label{lemma:kernel_impedance_bios}
We have \(\ker(\DD_{\kappa, \T_{\Sigma}}^{*})\neq\{0\}\)
if and only if the homogeneous \emph{Dirichlet} Helmholtz problem set in
\(\mathbb{R}^d \setminus \overline{\Omega}_{\text{B}}\) with wavenumber
\(\kappa\) has a non-trivial solution. The elements of
\(\ker(\DD_{\kappa, \T_{\Sigma}}^{*})\) are the Dirichlet traces of
solutions to Helmholtz problems set in \(\Omega_{\text{B}}\) with
non-homogeneous impedance boundary condition \(g \in
\ker(\V_{\kappa})\).
\end{lemma}

\begin{remark}\label{rm:kernels_equality}
  Since \(\Omega_{\text{B}}\) is a connected unbounded domain with bounded boundary, \(\ker(\V_{\kappa}) = \ker(\Id/2 + \Ktilde_{\kappa}) = \ker(\DD_{\kappa, \T_{\Sigma}})\).
  When \(\Omega_{\text{B}}\) is bounded, \(\ker(\V_{\kappa}) = \ker(\Id/2 + \Ktilde_{\kappa}) \neq \ker(\DD_{\kappa, \T_{\Sigma}})\).
\end{remark}

Because of spurious resonances, classical FEM-BEM formulations can be non-uniquely solvable at certain wavenumbers, even when~\eqref{pbm:initial} is well-posed, see~\cite{SchulzHiptmair2022SRC}.
The classical Johnson-Nédélec and Costabel couplings are of the form~\eqref{eq:var_pbm} by taking respectively
\begin{equation*}
  \A_{\Sigma,\text{JN}} \coloneqq
  \begin{bmatrix}
    0 & \;\Id \\
    \Id/2 - \K_{\kappa} &\; -\V_{\kappa}
  \end{bmatrix}
  \quad \text{ and } \quad
  \A_{\Sigma, \text{C}} \coloneqq
  \begin{bmatrix}
    \W_{\kappa} & \Id/2 + \Ktilde_{\kappa} \\
    \Id/2 - \K_{\kappa} & -\V_{\kappa}
  \end{bmatrix}.
\end{equation*}
We now establish that the GOSM reformulations~\eqref{eq:substructurePb}
of these FEM-BEM formulations suffer from the same
spurious resonances. Indeed, the inverse of \(\A_{\Sigma}- \imath
\B^*_{\Sigma} \T_{\Sigma} \B_{\Sigma}\) is needed to define
\(\Sop_{\text{B}}\), but, due to spurious resonances, \(\A_{\Sigma}- \imath
\B^*_{\Sigma} \T_{\Sigma} \B_{\Sigma}\) can be singular.  In the
following proofs, note that \(\B^*_{\Sigma}(p) = (p, 0)\), so
\(\B_{\Sigma}^* \T_{\Sigma} \B_{\Sigma}(\phi,p) = (\T_{\Sigma} \phi,
0)\).

\begin{proposition}[Johnson-Nédélec coupling]\label{prop:kernel_JN}
  If \(\A_{\Sigma} = \A_{\Sigma,\text{JN}}\), then
  \begin{equation*}
    \ker(\A_{\Sigma}- \imath \B^*_{\Sigma} \T_{\Sigma} \B_{\Sigma}) = 
    \left\{
    \left(
      \phi,
      \imath \T_{\Sigma} \phi
    \right)
    \mid    
    \phi \in \ker(\DD^*_{\kappa,\T_{\Sigma}})
    \right\}.
  \end{equation*}
\end{proposition}

\begin{proof}
  \((\phi, p) \in \ker(\A_{\Sigma}- \imath \B^*_{\Sigma} \T_{\Sigma} \B_{\Sigma})\) if and only if
  \( \lbrack\; - \imath \T_{\Sigma} \phi + p = 0 \;\text{and}\;
  (\Id/2 - \K_{\kappa}) \phi - \V_{\kappa} p = 0\;\rbrack\), which is equivalent to 
  \( \lbrack\;p = \imath \T_{\Sigma} \phi\;\text{and}\;
  \DD^*_{\kappa,\T_{\Sigma}} \phi = 0\;\rbrack  \).
\end{proof}

\begin{proposition}[Costabel coupling]\label{prop:kernel_Costabel}
  If \(\A_{\Sigma} = \A_{\Sigma, \text{C}}\), then
  \begin{equation*}
    \ker(\A_{\Sigma}- \imath \B^*_{\Sigma} \T_{\Sigma} \B_{\Sigma}) = 
    \left\{
    \left(
      0,
      p
    \right)
    \mid    
    p \in \ker(\DD_{\kappa,\T_{\Sigma}})
    \right\}.
  \end{equation*}
\end{proposition}

\begin{proof}
  We first prove (\(\subset\)). Let \((\phi, p) \in \ker(\A_{\Sigma}- \imath \B^*_{\Sigma} \T_{\Sigma} \B_{\Sigma})\). Then
  \begin{equation}\label{eq:calderon_costabel_1}
      (\W_{\kappa} - \imath \T_{\Sigma}) \phi + (\Id/2 + \Ktilde_{\kappa}) p = 0,\quad
      (\Id/2 - \K_{\kappa}) \phi - \V_{\kappa} p = 0.
  \end{equation}
  Let \(w \coloneqq \GL_{\Sigma}(\phi,p)\) and apply the Dirichlet-Neumann trace operator \(\gamma\) on \(w\) to obtain \(\traceDir(w) = \left(\Id/2 + \K_{\kappa}\right) \phi + \V_{\kappa} p\) and \(\traceNeu(w) = \W_{\kappa} \phi + \left(\Id/2 + \Ktilde_{\kappa}\right) p\).
  It follows from Equation~\eqref{eq:calderon_costabel_1} that \(\phi = \traceDir(w)\) and \(\traceNeu(w) = \imath \T_{\Sigma} \phi \). The impedance trace \(\traceNeu(w) - \imath \T_{\Sigma} \traceDir(w)\) is then null. Since by construction \(w\) is solution to the Helmholtz equation in \(\Omega_{\text{B}}\) and satisfies Sommerfeld's radiation condition, \(w = 0\), and so \(\phi = 0\). 
  Using \(\T_{\Sigma}\) to combine the two equations of~\eqref{eq:calderon_costabel_1}, we conclude that \(p\in \ker(\DD_{\kappa, \T_{\Sigma}})\).

  To prove (\(\supset\)), observe that according to \cref{rm:kernels_equality}, \(p\in \ker(\DD_{\kappa, \T_{\Sigma}})=\ker(\V_{\kappa})=\ker(\Id/2+\Ktilde_{\kappa})\), so we directly obtain the relations in Equation~\eqref{eq:calderon_costabel_1} for \(\phi = 0\).
\end{proof}

\begin{remark}
  Propositions~\ref{prop:kernel_JN}--\ref{prop:kernel_Costabel} still hold when \(\A_{\Sigma}\) is derived from a bounded domain, even though \cref{rm:kernels_equality} no more holds. The proofs are even simpler since \(\ker(\DD^*_{\kappa,\T_{\Sigma}})\) and \(\ker(\DD_{\kappa,\T_{\Sigma}})\) are both reduced to the trivial element, so is \(\ker(\A_{\Sigma}- \imath \B^*_{\Sigma} \T_{\Sigma} \B_{\Sigma})\).
\end{remark}

\begin{remark}
  Most of the restrictions imposed on the geometry in Section~\ref{sec:problem_definition} have been made to simplify the expression of the variational formulation~\eqref{eq:var_pbm}, and those of the kernels in Propositions~\ref{prop:kernel_JN}--\ref{prop:kernel_Costabel}. The expressions of the kernels hold when \(\Omega_B\) is still connected, but \(\Omega_{\text{F}}\) is not connected or \(\partial \Omega_{\text{O}} \nsubseteq \partial \Omega_{\text{F}}\) (see~e.g.~\figpart{fig:geometries}{right}), by taking care of considering \(\partial \Omega_B\) instead of \(\Sigma\).
\end{remark}

\section{Numerical illustration}

We consider the scattering of an incoming plane wave \(u_{i}(r,\theta) = \exp(\imath \kappa \, r \cos\theta)\), where \((r,\theta)\) are the polar coordinates, by a sound-soft obstacle \(\mathcal{D} = \Omega_{\text{O}}\), which is a disk of radius \(1\). 
The whole domain \(\Omega = \mathbb{R}^d \setminus \mathcal{D}\) is assumed homogeneous, that is, \(\kappa=k > 0\) constant.  
This can be modeled by Problem~\eqref{pbm:initial} with (non-homogeneous) Dirichlet boundary conditions and \(f=0\), whose unique solution is  
\begin{equation*}
u_{\text{S}}(r, \theta) \coloneqq -\sum_{p \in \mathbb{Z}} \exp(\imath p \theta)\imath^{\lvert p \rvert}
J_{\lvert p \rvert}(\kappa r) \frac{H^1_{\lvert p \rvert}(\kappa r)}{H^1_{\lvert p \rvert}(\kappa )},
\end{equation*}
with \(J_{\nu}\) and \(H^1_{\nu}\) respectively the Bessel and Hankel function of first kind of order \(\nu\).
Here, \(\Omega_{\text{F}}\) is the annulus of radii \(1\) and \(2\), and \(\Omega_{\text{B}} = \Omega \setminus (\overline{\Omega_{\text{F}}} \cup \overline{\Omega_{\text{O}}}) \), see Figure~\ref{fig:geometries}~(left).

As transmission operator \(\T_{\Sigma}\) in the BEM domain we choose
\(\W_{i \kappa}\), the hypersingular operator for the \emph{Yukawa
operator}, that is, \(-\Delta + \kappa^2\,\Id\).  As transmission
operator \(\T_{\Omega_{\text{F}}}\) in the FEM domain we choose a Schur
complement-based operator relying on the discretization of a positive
Dirichlet-to-Neumann map for the Yukawa operator,
see~\cite[Chapter~8]{Parolin2020NOD}.  For more details about the
expression of the chosen transmission operators, we refer
to~\cite[Sect.~10.2]{ourPaper}. We use $\mathbb{P}_1$-Lagrange
finite and boundary elements. Finite element (resp.~boundary element)
matrices are stored in sparse (resp.~dense) format. For the sake of
simplicity, for the discrete unknowns, we use the same notations as for
the continuous unknowns.  The numerical solution of the GOSM
substructured formulation~\eqref{eq:substructurePb} is obtained using
the {GMRes method}.

The spurious resonances for the considered geometry are the zeros of \(J_{\nu}\) divided by \(2\) (\(\Omega_{\text{F}}\) is of radius \(2\)), see e.g.~\cite[Th.~2.25]{ChandlerWildeGrahamEtAl2012NAB}. We assume that \(\kappa\) is not a spurious resonance, so \(\A_{\Sigma}- \imath \B^*_{\Sigma} \T_{\Sigma} \B_{\Sigma}\) is invertible. 
To illustrate the sensitivity or the robustness to spurious resonances of the GOSM for the Johnson-Nédélec and Costabel couplings, we study how the relative error of several quantities of interest evolves with respect to \(\kappa\).
We focus on the wavenumber range \(\kappa \in [4.28,4.42]\), inside which only \(\kappa_1\approx 4.32685\) and \(\kappa_2 \approx 4.38575\) are spurious resonances. All the experiments have been led with a fixed mesh generated for \(\kappa=10\) and \(20\) points per wavelength.
\begin{figure}
  \centering
  \includegraphics[scale=1]{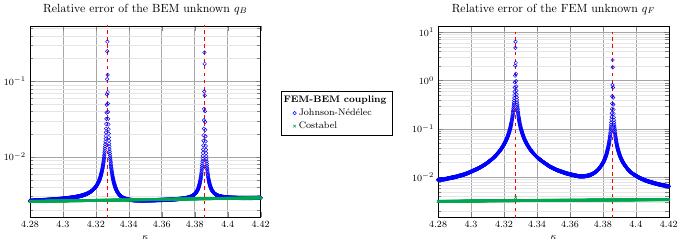}
  \caption{Relative error of the BEM (left) and FEM (right) unknowns.}\label{fig:bem_fem_relErr}
\end{figure}

Our first quantity of interest is the BEM unknown \(q_{\text{B}}\). In~\figpart{fig:bem_fem_relErr}{left} we observe, for the Johnson-Nédélec coupling, peaks in the relative error of \(q_{\text{B}}\) when \(\kappa\) is close to the spurious resonances \(\kappa_1\) and \(\kappa_2\). No peaks are observed in the curve associated with the Costabel coupling.
These results may be interpreted given the explicit expressions of the coupling kernels stated in Propositions~\ref{prop:kernel_JN}--\ref{prop:kernel_Costabel}, and of the scattering operator \(\Sop_{\text{B}}\) (see the end of Section~\ref{sec:problem_definition}).
First, assume we want to compute \(\Sop_{\text{B}} q\) for a given vector \(q\), and denote \(v_q \coloneqq (\A_{\Sigma}- \imath \B_{\Sigma}^* \T_{\Sigma} \B_{\Sigma})^{-1} \B_{\Sigma}^* q\). Because of the definition of \(\B_{\Sigma}\) only the first component of \(v_q\) is relevant to compute \(\Sop_{\text{B}} q = q + 2\, \imath \T_{\Sigma} \B_{\Sigma} v_q\).
Second, when \(\kappa\) is a spurious resonance only the second component of the eigenvectors of the Costabel operator \(\A_{\Sigma}- \imath \B^*_{\Sigma} \T_{\Sigma} \B_{\Sigma}\) is non-trivial. Thus, when numerically inverting that operator near a spurious resonance, we can expect the first component of \(v_q\) to be of good accuracy, while the second might be less accurate.
Turning to the eigenvectors of the Johnson-Nédélec operator, both components are non-trivial, so both components of \(v_q\) might deteriorate (especially the first one).

Next, we study the relative error of the FEM unknown \(q_{\text{F}}\), shown in \figpart{fig:bem_fem_relErr}{right}. As for the BEM curve, we observe peaks around spurious resonances only when considering the Johnson-Nédélec coupling. Nonetheless, we emphasize that \(\A_{\Omega_{\text{F}}}- \imath \B^*_{\Omega_{\text{F}}} \T_{\Omega_{\text{F}}} \B_{\Omega_{\text{F}}}\) is invertible, whatever the wavenumber. The reasoning used to explain the growth of the relative error of the BEM unknown \(q_{\text{B}}\) can not hold for \(q_{\text{F}}\).
This suggests that, near a spurious resonance, the poor quality of \(q_{\text{B}}\) makes also \(q_{\text{F}}\) less accurate. Remembering that \(q_{\text{F}}\) and \(q_{\text{B}}\) are the data shared between the FEM and BEM subdomains, that observation is not surprising.

\begin{figure}
  \includegraphics[scale=1]{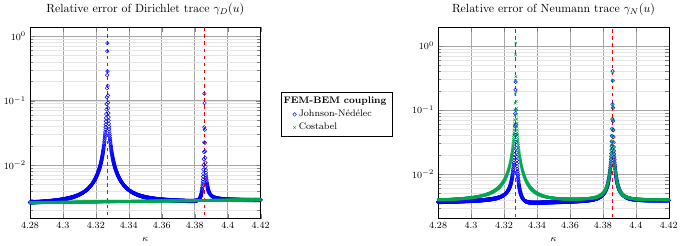}
  \caption{Relative errors of reconstructed Dirichlet (left) and Neumann (right) traces.}\label{fig:traces_relErr}
\end{figure}
We now turn to the evolution of the relative errors of \(\traceDir(u)\) and \(\traceNeu(u)\), with respect to \(\kappa\). 
We emphasize that \(\traceDir(u)\) and \(\traceNeu(u)\) are \emph{reconstructed} from the solution of the GOSM, namely 
\(
\left(
  \traceDir(u),
  \traceNeu(u)
\right)
  \coloneqq
  (\A_{\Sigma}- \imath \B_{\Sigma}^* \T_{\Sigma} \B_{\Sigma})^{-1} \B_{\Sigma}^*q_{\text{B}} = v_{q_{\text{B}}}.
\)
In \cref{fig:traces_relErr}, when \(\kappa\) is close to a spurious
resonance we observe peaks only in the Neumann relative error curve
for the Costabel coupling, while there are peaks in both Dirichlet and
Neumann curves for the Johnson-Nédélec coupling.  The results confirm
what we said previously about the deterioration near a spurious
resonance of the components of \(v_{q}\) for a given vector \(q\).  It
is interesting to note that the reconstructed solution
\((\traceDir(u), \traceNeu(u))\) is more and more spoiled when \(\kappa\) becomes closer to a spurious resonance, even though
\(\A_{\Sigma}- \imath \B^*_{\Sigma} \T_{\Sigma} \B_{\Sigma}\) is
invertible. We highlight that the relative errors go from
\(1\%\) to more than \(100\%\), but are not comparable: around
\(\kappa_1\) the Dirichlet relative error is close to \(1\), while the
Neumann relative error is close to \(0.4\).  These two errors are too
large for \(q_{\text{B}}\) and the reconstructed solution to be considered as
good quality approximations.

\begin{figure}
  \centering
  \includegraphics[scale=1]{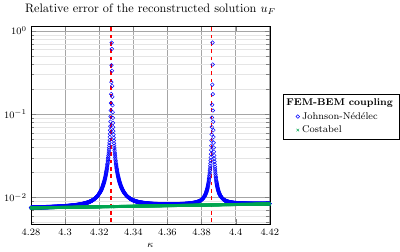}
  \caption{Relative error of the reconstructed solution \(u_{\text{F}}\) in the FEM domain.}\label{fig:fem_relErr}
\end{figure}
We end by looking at \cref{fig:fem_relErr}, which shows the relative error of \(u_{\text{F}}\), the reconstructed volume solution in \(\Omega_{\text{F}}\), with respect to \(\kappa\). We observe peaks around spurious resonances for the Johnson-Nédélec coupling, which are consequences of the peaks observed for \(q_{\text{F}}\) in \figpart{fig:bem_fem_relErr}{right}.
Moreover, it is coherent to recover an error for the solution in \(\Omega_{\text{F}}\) for the Johnson-Nédélec coupling, because its Dirichlet trace should be equal to \(\traceDir(u)\), for which we have also observed a peak in the relative error. 
On the other hand, for the same reason, it is coherent to not observe any peak in the relative error when the Costabel coupling is used. 

Finally, we pinpoint that relative errors peaks arise in the same way when classical FEM-BEM couplings~\eqref{eq:var_pbm} are considered. Indeed, for a geometry with two subdomains without an obstacle, it has been proven that only the Neumann component of the Costabel coupling can be spoiled, see for instance~\cite[Th.~1]{SchulzHiptmair2022SRC}, while for the Johnson-Nédélec coupling both volume and Neumann components can be spoiled.

\ethics{Acknowledgements}{This work is funded by the Inria program “Actions exploratoires” (OptiGPR3D).}
\vspace*{-0.7cm}
\input{references}
\end{document}

%% file: raccourcis.tex
\newcommand{\nc}{\newcommand}

\nc{\symbvec}[1]{\boldsymbol{\bm{#1}}} 

\nc{\bfx}{\mathbf{x}} 
\nc{\bfy}{\mathbf{y}} 
\nc{\bfz}{\mathbf{z}} 
\nc{\bfu}{\mathbf{u}} 
\nc{\bfv}{\mathbf{v}} 
\nc{\bfw}{\mathbf{w}} 
\nc{\bft}{\mathbf{t}} 
\nc{\bfb}{\mathbf{b}} 
\nc{\bfn}{\mathbf{n}} 
\nc{\bff}{\mathbf{f}} 
\nc{\bfq}{\mathbf{q}} %
\nc{\bfd}{\mathbf{d}} 
\nc{\bfe}{\mathbf{e}} 

\nc{\bfA}{\mathbf{A}} 
\nc{\bfR}{\mathbf{R}} 
\nc{\bfD}{\mathbf{D}} 
\nc{\bfI}{\mathbf{I}} 
\nc{\bfM}{\mathbf{M}} 
\nc{\bfK}{\mathbf{K}} 
\nc{\bfV}{\mathbf{V}} 
\nc{\bfW}{\mathbf{W}} 
\nc{\bfC}{\mathbf{C}} 
\nc{\bfP}{\mathbf{P}} 
\nc{\bfZ}{\mathbf{Z}} 
\nc{\bfF}{\mathbf{F}} 
\nc{\bfQ}{\mathbf{Q}} %
\nc{\bfU}{\mathbf{U}} %
\nc{\bfB}{\mathbf{B}} %
\nc{\bfH}{\mathbf{H}} %
\nc{\bfE}{\mathbf{E}} %
\nc{\bfS}{\mathbf{S}} %
\nc{\bfX}{\mathbf{X}} %

\nc{\bfId}{\mathbf{I}_{\mathrm{d}}} 

\nc{\hatu}{\hat{u}} 
\nc{\hatv}{\hat{v}} 

\nc{\bbN}{\mathbb{N}} 
\nc{\bbR}{\mathbb{R}} 
\nc{\bbC}{\mathbb{C}} 
\nc{\bbP}{\mathbb{P}} 
\nc{\bbH}{\mathbb{H}} 
\nc{\bbA}{\mathbb{A}} %
\nc{\bbT}{\mathbb{T}} %
\nc{\bbD}{\mathbb{D}} %
\nc{\bbJ}{\mathbb{J}} %
\nc{\bbF}{\mathbb{F}} %
\nc{\bbM}{\mathbb{M}} %
\nc{\bbQ}{\mathbb{Q}} %
\nc{\bbK}{\mathbb{K}} %

\nc{\wH}{\widetilde{H}}
\nc{\wA}{\widetilde{A}}

\nc{\calS}{\mathcal{S}} 
\nc{\calD}{\mathcal{D}} 
\nc{\calT}{\mathcal{T}} 
\nc{\calR}{\mathcal{R}} 
\nc{\calP}{\mathcal{P}} 
\nc{\calV}{\mathcal{V}} 
\nc{\calW}{\mathcal{W}} 
\nc{\calJ}{\mathcal{J}} 
\nc{\calL}{\mathcal{L}} 
\nc{\calH}{\mathcal{H}} 
\nc{\calO}{\mathcal{O}} 
\nc{\calE}{\mathcal{E}} 
\nc{\calI}{\mathcal{I}} 

\nc{\rmP}{\mathrm{P}}
\nc{\rmI}{\mathrm{I}}
\nc{\rmG}{\mathrm{G}}
\nc{\rmA}{\mathrm{A}}
\nc{\rmR}{\mathrm{R}}
\nc{\rmT}{\mathrm{T}}
\nc{\rmV}{\mathrm{V}}
\nc{\rmW}{\mathrm{W}}
\nc{\rmK}{\mathrm{K}}
\nc{\rmL}{\mathrm{L}}
\nc{\rmH}{\mathrm{H}}
\nc{\rmX}{\mathrm{X}}
\nc{\rmY}{\mathrm{Y}}
\nc{\rmS}{\mathrm{S}}
\nc{\rmD}{\mathrm{D}}

\nc{\fraku}{\mathfrak{u}}
\nc{\frakv}{\mathfrak{v}}
\nc{\frake}{\mathfrak{e}}

\nc\diff{\mathop{}\!\mathrm{d}}   
\DeclareMathOperator{\diag}{diag} 

\DeclareMathOperator{\Id}{Id} 

\nc{\bfPi}{\bm{\Pi}} 
\nc{\bfSigma}{\bm{\Sigma}} %

\DeclareMathOperator{\GL}{G} 
\DeclareMathOperator{\V}{V} 
\DeclareMathOperator{\K}{K} 
\DeclareMathOperator{\Ktilde}{\tilde{\K}} 
\DeclareMathOperator{\W}{W} 
\DeclareMathOperator{\A}{A} %
\DeclareMathOperator{\B}{B} %
\DeclareMathOperator{\DD}{D} %
\DeclareMathOperator{\Sop}{S} %
\DeclareMathOperator{\Piop}{\Pi} %
\DeclareMathOperator{\T}{T} 
\let\Re\relax
\DeclareMathOperator{\Re}{Re} 

\nc{\SumFourier}{\sum_{m \in \mathbb{Z}}} 

\newcommand{\traceNeu}[1][]{\gamma_N^{#1}}
\newcommand{\traceDir}[1][]{\gamma_D^{#1}}

\nc{\AnsCombined}{Ansatz-Combined}
\nc{\AnsDL}{Ansatz-DL}
\nc{\AnsSL}{Ansatz-SL}
\nc{\DirCombined}{Direct-Combined}
\nc{\DirFK}{Direct-FK}
\nc{\DirSK}{Direct-SK}
\nc{\DirTraceDir}{Direct-TrDir}
\nc{\DirTraceNeu}{Direct-TrNeu}

\nc{\mH}{\mathrm{H}}
\nc{\mX}{\mathrm{X}}
\nc{\setMI}{\mathcal{S}}
\nc{\bx}{\symbvec{x}} 
\nc{\by}{\symbvec{y}} 

\makeatletter
\DeclareFontFamily{OMX}{MnSymbolE}{}
\DeclareSymbolFont{MnLargeSymbols}{OMX}{MnSymbolE}{m}{n}
\SetSymbolFont{MnLargeSymbols}{bold}{OMX}{MnSymbolE}{b}{n}
\DeclareFontShape{OMX}{MnSymbolE}{m}{n}{
    <-6>  MnSymbolE5
   <6-7>  MnSymbolE6
   <7-8>  MnSymbolE7
   <8-9>  MnSymbolE8
   <9-10> MnSymbolE9
  <10-12> MnSymbolE10
  <12->   MnSymbolE12
}{}
\DeclareFontShape{OMX}{MnSymbolE}{b}{n}{
    <-6>  MnSymbolE-Bold5
   <6-7>  MnSymbolE-Bold6
   <7-8>  MnSymbolE-Bold7
   <8-9>  MnSymbolE-Bold8
   <9-10> MnSymbolE-Bold9
  <10-12> MnSymbolE-Bold10
  <12->   MnSymbolE-Bold12
}{}

\let\llangle\@undefined
\let\rrangle\@undefined
\DeclareMathDelimiter{\llangle}{\mathopen}%
                     {MnLargeSymbols}{'164}{MnLargeSymbols}{'164}
\DeclareMathDelimiter{\rrangle}{\mathclose}%
                     {MnLargeSymbols}{'171}{MnLargeSymbols}{'171}
\makeatother

\newcommand{\figpart}[2]{\cref{#1}~(#2)}

%% file: references.tex
%
%
\bibliographystyle{spmpsci}
\bibliography{BOISNEAULT_BONAZZOLI_MARCHAND_CLAEYS_bibliography.bib}

%% file: BOISNEAULT_BONAZZOLI_MARCHAND_CLAEYS_bibliography.bib
@article{BielakMacCamy1983EIP,
  author    = {J. Bielak and R. C. MacCamy},
  journal   = {Quarterly of Applied Mathematics},
  title     = {An exterior interface problem in two-dimensional elastodynamics},
  year      = {1983},
  number    = {1},
  pages     = {143--159},
  volume    = {41},
  doi       = {10.1090/qam/700668},
  publisher = {American Mathematical Society ({AMS})}
}

@book{SauterSchwab2011BEM,
  author     = {Sauter, S. A. and Schwab, C.},
  publisher  = {Springer-Verlag, Berlin},
  title      = {Boundary element methods},
  year       = {2011},
  isbn       = {978-3-540-68092-5},
  series     = {Springer Series in Computational Mathematics},
  volume     = {39},
  doi        = {10.1007/978-3-540-68093-2},
  mrclass    = {65-02 (35J25 46N40 47G10 65N38)},
  mrnumber   = {2743235},
  mrreviewer = {Paul Andrew Martin},
  pages      = {xviii+561}
}

@incollection{Costabel1987SMC,
  author    = {M. Costabel},
  booktitle = {Mathematical and Computational Aspects},
  publisher = {Springer Berlin Heidelberg},
  title     = {Symmetric Methods for the Coupling of Finite Elements and Boundary Elements},
  year      = {1987},
  pages     = {411--420},
  doi       = {10.1007/978-3-662-21908-9_26}
}

@article{HiptmairMeury2006SFB,
  author    = {R. Hiptmair and P. Meury},
  journal   = {{SIAM} Journal on Numerical Analysis},
  title     = {Stabilized {FEM}-{BEM} Coupling for {H}elmholtz Transmission Problems},
  year      = {2006},
  number    = {5},
  pages     = {2107--2130},
  volume    = {44},
  doi       = {10.1137/050639958},
  publisher = {Society for Industrial {\&} Applied Mathematics ({SIAM})}
}

@article{JohnsonNedelec1980CBI,
  author    = {C. Johnson and J.-C. N{\'{e}}d{\'{e}}lec},
  journal   = {Math. of Comput.},
  title     = {On the coupling of boundary integral and finite element methods},
  year      = {1980},
  number    = {152},
  pages     = {1063--1079},
  volume    = {35},
  doi       = {10.1090/s0025-5718-1980-0583487-9},
  publisher = {American Mathematical Society ({AMS})}
}

@book {MR1822275,
    AUTHOR = {N\'ed\'elec, J.-C.},
     TITLE = {Acoustic and electromagnetic equations},
    SERIES = {Applied Mathematical Sciences},
    VOLUME = {144},
      NOTE = {Integral representations for harmonic problems},
 PUBLISHER = {Springer-Verlag, New York},
      YEAR = {2001},
     PAGES = {x+316},
      ISBN = {0-387-95155-5},
   MRCLASS = {35-02 (35C15 35J05 35Q60 45H05 78A25)},
  MRNUMBER = {1822275},
MRREVIEWER = {Rainer\ Picard},
       DOI = {10.1007/978-1-4757-4393-7},
}

@article{ChandlerWildeGrahamEtAl2012NAB,
  author    = {S. Chandler-Wilde and I. Graham and S. Langdon and E. Spence},
  journal   = {Acta Numerica},
  title     = {Numerical-asymptotic boundary integral methods in high-frequency acoustic scattering},
  year      = {2012},
  pages     = {89--305},
  volume    = {21},
  doi       = {10.1017/s0962492912000037},
  publisher = {Cambridge University Press ({CUP})}
}

@Book{McLean2000SES,
  author    = {McLean, W. C. H.},
  publisher = {Cambridge University Press},
  title     = {Strongly elliptic systems and boundary integral equations},
  year      = {2000},
  isbn      = {0521663326},
  pages     = {357},
}

@Article{SchulzHiptmair2022SRC,
  author    = {Schulz, E. and Hiptmair, R.},
  journal   = {Computational Methods in Applied Mathematics},
  title     = {Spurious Resonances in Coupled Domain-Boundary Variational Formulations of Transmission Problems in Electromagnetism and Acoustics},
  year      = {2022},
  issn      = {1609-9389},
  number    = {4},
  pages     = {971--985},
  volume    = {22},
  doi       = {10.1515/cmam-2021-0197},
  publisher = {Walter de Gruyter GmbH},
}

@PhdThesis{Parolin2020NOD,
  author      = {Parolin, E.},
  school      = {{IP Paris}},
  title       = {{Non-overlapping domain decomposition methods with non-local transmission operators for harmonic wave propagation problems}},
  year        = {2020},
  type        = {PhD thesis},
  hal_id      = {tel-03118712},
  hal_version = {v1},
  number      = {2020IPPAE011},
  url         = {https://theses.hal.science/tel-03118712},
}

@unpublished{ourPaper,
  author = {A. Boisneault and M. Bonazzoli and X. Claeys and P. Marchand},
  title = {Discrete {FEM-BEM} coupling with a {G}eneralized {O}ptimized {S}chwarz {M}ethod},
  note = {In preparation},
}

@article {Claeys2023NOS,
    AUTHOR = {Claeys, X.},
     TITLE = {Nonlocal optimized {S}chwarz method for the {H}elmholtz
              equation with physical boundaries},
   JOURNAL = {SIAM J. Math. Anal.},
  FJOURNAL = {SIAM Journal on Mathematical Analysis},
    VOLUME = {55},
      YEAR = {2023},
    NUMBER = {6},
     PAGES = {7490--7512},
      ISSN = {0036-1410,1095-7154},
   MRCLASS = {65N55 (31B10 35J10 65N38)},
  MRNUMBER = {4665035},
       DOI = {10.1137/23M1545847},
}
